\pdfoutput=1
\RequirePackage{ifpdf}
\ifpdf 
\documentclass[pdftex]{sigma}
\else
\documentclass{sigma}
\fi

\numberwithin{equation}{section}

\usepackage{bm}

\usepackage{ytableau}

\newtheorem{Theorem}{Theorem}[section]
\newtheorem*{Theorem*}{Theorem}

\newtheorem{Lemma}[Theorem]{Lemma}
\newtheorem{Proposition}[Theorem]{Proposition}

\theoremstyle{definition}
\newtheorem{Definition}[Theorem]{Definition}

\newtheorem{Example}[Theorem]{Example}
\newtheorem{Remark}[Theorem]{Remark}

\newcommand{\UU}{\mathrm{U}}

\newcommand{\trans}[1]{#1^{\mathtt{T}}}
\newcommand{\fix}[1]{\mathsf{f}(#1)}

\newcommand{\Refl}{\mathcal{R}}
\newcommand{\mult}{\mathsf{m}}

\DeclareMathOperator{\Raise}{as}
\DeclareMathOperator{\Lower}{ds}

\DeclareMathOperator{\sgn}{sgn}

\DeclareMathOperator{\G}{G}
\DeclareMathOperator{\Wg}{Wg}
\DeclareMathOperator{\Fix}{Fix}

\begin{document}

\allowdisplaybreaks

\newcommand{\arXivNumber}{2510.21186}

\renewcommand{\PaperNumber}{077}

\FirstPageHeading

\ShortArticleName{Weingarten Calculus with Virtual Isometries}

\ArticleName{Weingarten Calculus with Virtual Isometries}

\Author{Beno\^{\i}t COLLINS~$^{\rm a}$ and Sho MATSUMOTO~$^{\rm b}$}

\AuthorNameForHeading{B.~Collins and S.~Matsumoto}

\Address{$^{\rm a)}$~Department of Mathematics, Kyoto University, Japan}
\EmailD{\mail{collins@math.kyoto-u.ac.jp}}
\URLaddressD{\url{https://benoitcollins.github.io/}}

\Address{$^{\rm b)}$~Graduate School of Science and Engineering, Kagoshima University, Japan}
\EmailD{\mail{shom@sci.kagoshima-u.ac.jp}}
\URLaddressD{\url{https://sites.google.com/site/shomatsumotomath1/}}

\ArticleDates{Received February 24, 2026, in final form August 10, 2026; Published online August 18, 2026}

\Abstract{In this paper, we develop a novel approach to the Weingarten calculus by employing the notion of virtual isometries. Traditionally, Weingarten calculus provides explicit formulas for integrating polynomial functions over compact matrix groups with respect to the Haar measure, yet it faces limitations when evaluating high-degree integrals due to the non-invertibility of the associated matrices. We revisit these classical computations from a new perspective: by constructing Haar-distributed matrices as products of sequences of complex reflections, we derive new recursive structures for the Weingarten functions across different dimensions. This framework leads to two main results: (1) an explicit Weingarten calculus for complex reflections, yielding systematic moment computations for associated rank-one matrices, and (2) a novel convolution formula that connects Weingarten functions in dimension $n$ to those in dimension $n-1$, through the introduction of ascension functions in the symmetric group algebra. Our approach not only provides a unified treatment for unitary groups, but also sheds light on the algebraic and probabilistic aspects of high-degree integral computations. We present several examples and applications.}


\Classification{60B20; 15B52; 28C10}

\section{Introduction}

\subsection{Weingarten calculus}

Weingarten calculus is a method that allows one to explicitly calculate integrals of polynomial functions on compact groups with respect to the Haar probability measure.

An important case, which we focus on in this paper, is the unitary group
$\UU(n)$, which consists of all $n \times n$ unitary matrices.
Let $U=(u_{ij})_{1 \le i,j \le n}$ be a Haar-distributed matrix from $\UU(n)$.
By properties of the Haar measure, when $k$ and $l$ are distinct nonnegative integers, the expectation of the following expression is zero:
\[
\mathbb{E} \bigl[ u_{i_1 j_1} u_{i_2 j_2} \cdots u_{i_k j_k}
\overline{u_{i_1' j_1'} u_{i_2' j_2'}\cdots u_{i_l' j_l'}} \bigr],
\]
where $i_p,j_p \in \{1,2,\dots,n\}$ for $1 \le p \le k$, and $i_p',j_p' \in \{1,2,\dots,n\}$ for $1 \le p \le l$.
According to the \emph{Weingarten formula} \cite{Collins2003, CollinsSniady2006},
when $k=l$, this moment can be expressed as a sum over the symmetric group $S_k$:
\begin{gather}
 \mathbb{E} \bigl[ u_{i_1 j_1} u_{i_2 j_2} \cdots u_{i_k j_k}
\overline{u_{i_1' j_1'} u_{i_2' j_2'}\cdots u_{i_k' j_k'}} \bigr] \nonumber\\
\qquad =\sum_{\sigma \in S_k}\sum_{\tau\in S_k}
\delta(i_{\sigma(1)}, i_1')\cdots \delta(i_{\sigma(k)}, i_k')
\delta(j_{\tau(1)}, j_1')\cdots \delta(j_{\tau(k)}, j_k')
\Wg_{k,n}\bigl(\sigma \tau^{-1}\bigr).\label{eq:Wg_F}
\end{gather}
Here $\delta(p,q)$ is the Kronecker delta, and $\Wg_{k,n}$ is a class function on $S_k$ known as the (unitary)
\emph{Weingarten function}.
See also \cite{MR597583}.
The Weingarten function has various representations; however, when $k \le n$, we can
immediately obtain the representation
 \[
 \Wg_{k,n}(\sigma) = \mathbb{E} \bigl[ u_{11} u_{22} \cdots u_{kk}
 \overline{u_{\sigma(1)1} u_{\sigma(2)2} \cdots u_{\sigma(k)k}} \bigr]
 \]
from \eqref{eq:Wg_F}.

The Weingarten function has been the subject of extensive research. This includes expansions using characters of symmetric groups, expressions through Jucys--Murphy elements, and series expansions as generating functions for monotone factorizations of permutations \cite{MatsumotoNovak, Novak2010}.
In addition, a relation between the Weingarten functions on $S_k$ and those on $S_{k-1}$ has been established \cite{CollinsMatsumoto2017}.

\subsection{The problem of high degree integrals}\label{sec:problem}

In the Weingarten formula \eqref{eq:Wg_F}, the Weingarten function $\Wg_{k,n}$, or
the Weingarten matrix \smash{$\bigl(\Wg_{k,n}\bigl(\sigma \tau^{-1}\bigr)\bigr)_{\sigma,\tau \in S_k}$}
is given by the inverse of the Gram matrix \smash{$\G_{k,n}=\bigl(n^{\kappa(\sigma\tau^{-1})}\bigr)_{\sigma,\tau \in S_k}$}, where $\kappa(\cdot)$ denotes the number of cycles of a permutation.
However, when $k>n$, that is, when evaluating integrals of degree $k$ strictly greater than $n$ on the unitary group of dimension $n$, the matrix $\G_{k,n}$ is no longer invertible. There are several approaches to address this problem. The first approach is to take the pseudo-inverse of the Gram matrix when it is not invertible. It is important to note that while a Weingarten formula exists irrespective of the degree, the Weingarten function is not unique when the degree exceeds the dimension. For details, see, for example, \cite{CMN} or
Section~\ref{subsec:low_dim} in the present article.

We are aware of two additional routes to compute polynomial integrals of high degree against the Haar measure, and describe them briefly in what follows. Both workarounds are successful from a theoretical point of view in the sense that they allow us to compute the integrals in principle. However, at this point, we are unable to extract any useful information for concrete applications.
The first attempt can be found as a remark of a paper by Collins and \'Sniady~\cite{CollinsSniady2006}, and the authors, along with Fukuda, have a joint working project to expand on this topic~\cite{CollinsFukudaMatsumoto}.
Roughly, the idea is that any formal computation in the dimension, while in principle only well defined for a dimension that is large enough -- and potentially larger than the dimension in which one wants to compute, doing the calculation formally, simplifying, and then evaluating in the target dimension will systematically give the correct result.

The second route was taken by Cioppa and Collins in the
unpublished article~\cite{CioppaCollins}.
Here, the idea is to replace the indicator functions which are simple to compute but do not yield an orthogonal Gram matrix by much more complicated functions obtained from the Gelfand--Zetlin basis, i.e., to exhibit matrix elements of the commutant of tensor representations of the unitary group.

The approach developed in the present paper is distinct from the three approaches recalled above; it is actually the motivation for this paper. It comes from the following simple idea: from the point of view of differential geometry and measure theory,
$\UU(n)/\UU(n-1)=S^{2n-1}$, and thus
 $\UU(n)=S^{2n-1}\times \UU(n-1)$
 and this extends to the level of product of uniform measures.
 For a~rigorous statement, we refer to Lemma~\ref{lem:BNN}.
 Therefore, we can understand the Haar measure at the level $n$ from the Haar measure at the level $n-1$ and the uniform measure on the sphere. But it turns out that the uniform measure on the sphere can be understood with arbitrary order directly thanks to a Gaussianization trick.
 The Gaussianization trick alone was already used in~\cite{BanicaCollins2008}; the novelty here is its systematic combination with the inductive decomposition of the Haar measure into complex reflections, which yields closed-form reflection moments (Section~\ref{section:Wg_P}) and new recursion formulas for the Weingarten function (Section~\ref{section:Wg_recursive}). We also refer to the remark after Lemma~\ref{lem:Wg_for_R} for more details.

 This is the approach that we develop in this paper and it allows us to revisit the theory of Weingarten calculus from a completely different point of view
 and, in principle, to solve completely the problem mentioned at the beginning of this subsection.

Let us state precisely the scope of our solution to the high-degree problem.
Theorem~\ref{thm:momentsP} computes the moments of the matrix entries of a random complex reflection in closed form, with no restriction on the degree of the monomial nor on the dimension $n$.
Through the decomposition of Lemma~\ref{lem:BNN_g}, every polynomial integral against the Haar measure on $\UU(n)$ --- of arbitrary degree, in particular of degree larger than $n$ --- reduces in finitely many algebraic steps to such reflection moments; Section~\ref{sec:compute-integral} describes this reduction, and Example~\ref{ex:U2_high_degree} carries it out for a~family of integrals of arbitrarily high degree in dimension $2$.
We emphasize that what we obtain is an algorithm that terminates for every given integral, rather than a single closed formula: the new recursive formula for the Weingarten function itself (Theorem~\ref{thm:recursive}) requires $k<n$, with a~variant for $k=n$ (Theorem~\ref{thm:dim_equal_degree}).

Let us also mention that integrals of degree larger than the dimension arise naturally in applications: for instance, in lattice gauge theory, where strong-coupling expansions require moments of arbitrarily high order over $\UU(n)$ for a fixed small $n$ (see \cite{MR597583}), and in quantum information theory, where one considers many copies $U^{\otimes k}$ of a random unitary acting on a~system of fixed low dimension.

For that purpose, we need the concept of virtual isometries, which we describe in the forthcoming subsection.

\subsection{Virtual isometries}

Olshanski \cite{Olshanski_unitary} considered a space of \emph{virtual unitary matrices} as projective limits of unitary groups.
Additionally, Bourgade, Najnudel, and Nikeghbali \cite{BNN} independently explored this concept and referred to it as \emph{virtual isometry}.
In this paper, we will present an introduction based on the latter approach.
 The term\ ``virtual'' is used because these objects arise as projective limits of sequences of unitary matrices.

A virtual isometry is a sequence $(g_n)_{n \ge 1}$ of unitary matrices
such that for all $n \ge 1$, $g_n \in \UU(n)$ and $\pi_{n}(g_{n})=g_{n-1}$.
Here, a specific definition for the projection $\pi_{n}\colon \UU(n) \to \UU(n-1)$ is given by Neretin \cite{Neretin} as follows.
For a unitary matrix $g_{n}= (a_{ij})_{1 \le i,j \le n} \in \UU(n)$,
\[
\pi_{n} (g_{n})= \left(a_{ij} + \frac{a_{in} a_{nj}}{1-a_{nn}} \right)_{1 \le i,j \le n-1}
\]
if $a_{nn} \not=1$,
and $\pi_{n} (g_{n})=(a_{ij})_{1 \le i,j \le n-1}$ if $a_{nn}=1$.
Such virtual isometries serve as extensions of the virtual permutations first presented by Kerov, Olshanski, and Vershik \cite{KOV1993, KOV2004}.

Due to \cite{BNN} (see also \cite{MNN}),
virtual isometries can be constructed inductively as follows:
We equip $\mathbb{C}^n$ with the canonical scalar product: $\langle \bm{v},\bm{w} \rangle
=\sum_{j=1}^n v_j \overline{w_j}$ for $\bm{v}= \trans{(v_1,\dots,v_n)}$,
$\bm{w}=\trans{(w_1,\dots,w_n)} \in \mathbb{C}^n$.
Here, $\trans{A}$ and $\trans{\bm{v}}$ represent the transpose of a matrix $A$ or a vector $\bm{v}$.
Let $\bm{e}_1,\dots, \bm{e}_n$ be the standard orthonormal basis of $\mathbb{C}^n$.
Consider a sequence $(\bm{x}_n)_{n \ge 1}$ of column vectors, where each $\bm{x}_n$ lies on
the complex unit sphere of $\mathbb{C}^n$ for $n \ge 1$.
We define a sequence of unitary matrices $r_n \in \UU(n)$ in the following way:
\begin{itemize}\itemsep=0pt
\item $r_n$ is the identity matrix of size $n$ if $\bm{x}_n=\bm{e}_n$.
\item Otherwise, $r_n \in \UU(n)$ is the unique complex reflection such that $r_n \bm{e}_n=\bm{x}_n$.
\end{itemize}
A virtual isometry $(g_n)_{n \ge 1}$ is given by induction as follows:{\samepage
\[
g_1:=r_1=\bm{x}_1, \qquad g_n:= r_n \begin{pmatrix} g_{n-1} & \bm{0} \\ \trans{\bm{0}} & 1 \end{pmatrix}
\qquad \text{for $n \ge 2$}.
\]
One can check that $\pi_{n}(g_n)=g_{n-1}$ in this construction.}

The virtual isometry is consistent with the Haar measure.
Indeed, let $(X_n)_{n \ge 1}$ be a sequence of independent random vectors, each uniformly distributed on the complex unit sphere of $\mathbb{C}^n$.
If we denote by $R_n$ and $U_n$ the corresponding random matrices $r_n$ and $g_n$, respectively,
in the above constructions, then~$U_n$ follows the Haar measure on~$\UU(n)$.
 See \cite[Proposition~3.1]{BNN} for details.

A similar construction method for the Haar measure was also provided by A.~Hurwitz (see~\cite{DiaconisForrester}).
Virtual isometries have proven to be valuable in the analytic aspects of random matrix theory, enabling the attainment of almost sure results due to the correlations they provide between sequences of Haar-distributed unitary matrices (see \cite{Assiotis, CNN}).

\subsection{Main results}

Our goal is to provide a new perspective on Weingarten calculus based on the construction method of virtual isometry.
The two main results of this paper are as follows.

The first main result is Theorem~\ref{thm:momentsP}, which focuses on a random complex reflection $R_n$ used in a construction involving virtual isometries.
We establish Weingarten calculus for $R_n$, which provides a systematic method for computing the moments of the matrix elements of $R_n$.
Instead of working directly with $R_n$,
we will deal with $P_n:=I_n-R_n$ because $P_n$ is a rank-one matrix, making it easier to handle.
Fortunately, the Weingarten calculus for $P_n=(p_{ij})_{1 \le i,j \le n}$ has been established in a concise form.
In fact, the moment for the form
\[
\mathbb{E} \bigl[ p_{i_1 j_1} p_{i_2 j_2} \cdots p_{i_m j_m}
\overline{p_{i_1' j_1'} p_{i_2' j_2'} \cdots p_{i_l' j_l'}}
\bigr]
\]
is explicitly given by a simple expression.
We provide many examples of the calculation of the moments of matrix elements of $R_n$ and $P_n$.
For example, we obtain
\[
\mathbb{E} \bigl[|p_{11} p_{22} \cdots p_{kk}|^2\bigr] = \frac{2^k}{n^2(n+1)^2 \cdots (n+k-1)^2}
\]
if $k<n$.
From the construction of the virtual isometry, the moments for $P_n$ essentially contain all the information of a Haar-distributed unitary matrix $U_n$.
Therefore, this provides a theoretical method to calculate the moments of Haar-distributed unitary matrices that differs from the traditional Weingarten calculus.

Our second main result is Theorem~\ref{thm:recursive}, in which an explicit relationship is established between the unitary Weingarten functions of dimensions $n$ and $n-1$.
More specifically, if $n >k$, the Weingarten function $\Wg_{k,n}$ on the symmetric group $S_k$ is expressed as a convolution product in the group algebra of $S_k$:
\[
\Wg_{k,n}= \Raise_{k,n} * \Wg_{k,n-1},
\]
where $\Raise_{k,n}$ is a new class function on $S_k$ and $\Wg_{k,n-1}$ is the Weingarten function of lower dimension $n-1$.
We call the function $\Raise_{k,n}$ an \emph{ascension function} and investigate its basic properties.
For example, finding the inverse of the function $\Raise_{k,n}$ with respect to the convolution product in the symmetric group algebra is straightforward.
The assumption $n > k$ is necessary in the above formula. Although a similar expression can be derived for the case when $ n = k $ with minor adjustments, it is important to emphasize that the resulting expression does not apply universally for all values of $ n $ and $ k $. Consequently, we obtain a new representation of the Weingarten function utilizing the ascension functions and related constructs (see Proposition~\ref{prop:Wg_asasas}).

These two main results are discussed in Sections \ref{section:Wg_P} and \ref{section:Wg_recursive}, respectively.
Importantly, let us add that Section~\ref{sec:compute-integral} explains how to use these results systematically to compute integrals of polynomials of arbitrary degree, thus providing a new approach to the problem mentioned in Section~\ref{sec:problem}.

Extending these results to the orthogonal group is not particularly challenging; this will be presented at a later opportunity.

At the end of the introduction, we explain the notation used.
Throughout this article, we use the notation
\[
a^{\uparrow k} = \prod_{j=0}^{k-1} (a+j),
\]
where $a$ is a complex number and $k$ is a positive integer. For convenience, we also set $a^{\uparrow 0}=1$.
For clarity, we will use both $\delta_{ij}$ and $\delta(i,j)$ to denote the Kronecker delta.
The symbol $ I_n $ represents the identity matrix of size $ n $.

\section{Weingarten calculus for complex reflections} \label{section:Wg_P}

\subsection{Complex reflection in the construction of virtual isometry}

We equip $\mathbb{C}^n$ with the canonical scalar product: $\langle \bm{v},\bm{w} \rangle
=\sum_{j=1}^n v_j \overline{w_j}$ for $\bm{v}= \trans{(v_1,\dots,v_n)}$, $\bm{w}=\trans{(w_1,\dots,w_n)} \in \mathbb{C}^n$.
Let $\bm{e}_1,\dots, \bm{e}_n$ be the standard orthonormal basis of $\mathbb{C}^n$.

A complex reflection on $\mathbb{C}^n$ is a unitary transformation $\Refl$ such that it is either the identity map $\mathrm{Id}$ or the rank of $\mathrm{Id}-\Refl$ is $1$.
Every complex reflection can be represented as
\[
\Refl_{\bm{a},\alpha}( \bm{v})= \bm{v} - (1-\alpha) \frac{\langle \bm{v},\bm{a} \rangle}{\langle \bm{a},\bm{a} \rangle} \bm{a}, \qquad \bm{v} \in \mathbb{C}^n,
\]
where $\bm{a}$ is some nonzero vector and $\alpha$ is a complex number with $|\alpha|=1$.
 Observe that $\Refl_{\bm{a},\alpha}(\bm{a})= \alpha \bm{a}$ and
$\Refl_{\bm{a},\alpha}(\bm{w})= \bm{w}$ if $\bm{w}$ is orthogonal to $\bm{a}$.
Given two distinct unit vectors $\bm{e}$ and $\bm{m}$, there exists a unique reflection $\Refl$ such that
$\Refl(\bm{e})=\bm{m}$, and it is given by $\Refl_{\bm{m}-\bm{e},\alpha}$ with \smash{$\alpha= -\frac{1-\langle \bm{m},\bm{e} \rangle}{1-\langle \bm{e},\bm{m} \rangle}$}.

Let
\[
X=X_n =\trans{(x_1,x_2,\dots,x_n)}
\]
be a uniform random vector on the unit sphere of $\mathbb{C}^n$.
Since $X$ is almost surely different from~$\bm{e}_n$,
 there exists a unique reflection $\Refl$ such that
$\Refl(\bm{e}_n)=X$ as described above.
We denote by $R=R_n=(r_{ij})_{1 \le i,j \le n}$ its matrix representation with respect to the basis $\bm{e}_1,\dots, \bm{e}_n$.
It is straightforward to see that
\begin{equation} \label{matrixR}
r_{ij}= \begin{cases}
\delta_{ij}-\dfrac{x_{i}\overline{x_{j}}}{1-\overline{x_{n}}} &\text{if $j<n$ and $i <n$}, \\
\dfrac{1-x_{n}}{1-\overline{x_{n}}}\, \overline{x_{j}} &\text{if $j<n$ and $i=n$},\\
x_{i} & \text{if $j=n$ and $i \le n$}.
\end{cases}
\end{equation}
In particular, the last column of $R$ is $X_n$.

The following lemma is the key to our approach. For a proof, we refer to \cite[Proposition 3.1]{BNN} or \cite[Theorem 1]{BNR}.
See also \cite[Proposition 2.1]{BHNY}.

\begin{Lemma}
\label{lem:BNN}
Let $n \ge 2$. Let $V$ be a Haar-distributed unitary matrix from $\UU (n-1)$ and set
\begin{equation*} 
\tilde{V}:= V \oplus (1)=\begin{pmatrix} V & \bm{0} \\ \bm{0} & 1 \end {pmatrix} \in \UU(n).
\end{equation*}
Let $X=\trans{(x_1,x_2,\dots,x_n)}$ be a uniform random vector on the unit sphere of $\mathbb{C}^n$, independent of $V$,
and we construct the random unitary matrix $R=(r_{ij})$ as described in \eqref{matrixR}.
Then the product $U:=R \tilde{V}$ follows the Haar measure on $\UU (n)$.
\end{Lemma}

The next lemma will be used several times.

\begin{Lemma} \label{lem:Wg_for_R}
Let $X=\trans{(x_1,\dots,x_n)}$ be a uniform vector on the unit sphere of $\mathbb{C}^n$.
For $2n$ nonnegative integers $m_1,\dots, m_n, l_1,\dots, l_n$, the moment
\[
\mathbb{E}\bigl[x_1^{m_1} x_2^{m_2} \cdots x_n^{m_n} \overline{x_1^{l_1} x_2^{l_2} \cdots x_n^{l_n}} \bigr]
\]
 vanishes unless $m_j=l_j$ for all $1 \le j \le n$.
In this case, we have
\[
\mathbb{E} \bigl[|x_1|^{2m_1} |x_2|^{2m_2} \cdots |x_n|^{2m_n}\bigr]
= \frac{m_1! \, m_2! \, \cdots \, m_n!}{n^{\uparrow m}},
\]
where $m=m_1+ m_2+ \cdots+m_n$.
\end{Lemma}

\begin{proof}
The random vector $X$ has the same distribution as
the first column of a Haar-distributed unitary matrix from $\UU(n)$.
This lemma is an easy example of Weingarten calculus.
In fact, using~\eqref{eq:Wg_F}, we can see that
\begin{align*}
\mathbb{E} \bigl[|x_1|^{2m_1} |x_2|^{2m_2} \cdots |x_n|^{2m_n}\bigr]
&=m_1! m_2! \cdots m_n! \sum_{\sigma \in S_m} \Wg_{m,n}(\sigma) \\
&= \frac{m_1! m_2! \cdots \, m_n!}{n(n+1) \cdots (n+m-1)}.
\end{align*}
The second equality can be easily proved, for example, by using the character expansion~\eqref{eq:Wg_char} given later.
When $m > n$, formula~\eqref{eq:Wg_F} remains valid with $\Wg_{m,n}$ replaced by the pseudo-inverse Weingarten function of Section~\ref{subsec:low_dim}, and the identity $\sum_{\sigma \in S_m} \Wg_{m,n}(\sigma) = 1/n^{\uparrow m}$ continues to hold, as it corresponds to the one-row partition $(m)$, whose length never exceeds $n$. Alternatively, the Gaussianization argument of Remark~\ref{rem:gaussianization} below yields the statement directly for every $m$.
\end{proof}

\begin{Remark} \label{rem:gaussianization}
 This lemma can also be obtained through a trick that is sometimes
called the gaussianization method.
Indeed, let
$y_1,\dots, y_n$ be $n$ i.i.d standard complex Gaussian variables. Then, the random vector
$y/\|y\|$ has the same distribution as
the uniformly distributed vector
$x=(x_1,\dots, x_n)$ on the complex unit sphere.
Conversely, if $y$ and $x$ are independent,
$\|y\|\cdot x$ has the same distribution as $y$.
Using the complex version of Wick's theorem and the fact that~$\|y\|$ has a complex Chi-square distribution, thus satisfying
$\mathbb{E}\bigl(\|y \|^{2k}\bigr)=n(n+1)\cdots (n+k-1)$,
the proof follows.
For applications and related materials, we refer to~\cite[Theorem 6.1]{BanicaCollins2008}.
We also refer to~\cite{MR1936554}, in which a Chi-square distribution also appears in a tridiagonalization of a related random matrix.
\end{Remark}

\subsection{Weingarten calculus for a rank-one matrix}

Instead of working directly with $R$,
we deal with $P =(p_{ij})_{1 \le i,j \le n}:=I-R$ because $P$ is a~rank-one matrix, thus easier to handle.
By \eqref{matrixR}, we have
\begin{equation} \label{matrixP}
p_{ij}= \begin{cases}
\dfrac{x_{i}\overline{x_{j}}}{1-\overline{x_{n}}} &\text{if $j<n$ and $i <n$}, \\
-\dfrac{1-x_{n}}{1-\overline{x_{n}}}\, \overline{x_{j}} &\text{if $j<n$ and $i=n$},\\
-x_i & \text{if $j=n$ and $i < n$}, \\
1-x_n & \text{if $i=j=n$}.
\end{cases}
\end{equation}

Let us fix some notation.
For a sequence $\bm{i}=(i_1,\dots, i_m)$ of positive integers and a positive integer $k$,
we denote by $\mult_k(\bm{i})$ the multiplicity of $k$ in $\bm{i}$:
\[
\mult_k (\bm{i})= |\{h \in \{1,2,\dots, m\} \mid i_h=k\}|.
\]
For two sequences
$\bm{i}=(i_1,\dots, i_m)$ and $\bm{j}=(j_1,\dots, j_l)$,
we define a combined sequence
$\bm{i} \sqcup \bm{j} = (i_1,\dots, i_m, j_1,\dots,j_l)$.
In what follows, the order of the letters is not significant.
Note that $\mult_k(\bm{i} \sqcup \bm{j})= \mult_k(\bm{i})+ \mult_k(\bm{j})$.

The following theorem is our first main result.

\begin{Theorem} \label{thm:momentsP}
Let $m$, $l$ be nonnegative integers, and consider four sequences
\[
\bm{i}=(i_1,\dots, i_m), \qquad \bm{j} =(j_1,\dots, j_m), \qquad
\bm{i}'=(i_1',\dots, i_l'), \qquad \bm{j}' =(j_1',\dots, j_l').
\]
of positive integers in $\{1,\dots, n\}$.
Then the moment
\[
\mathbb{E} \bigl[ p_{i_1 j_1} \cdots p_{i_m j_m} \overline{p_{i_1' j_1'} \cdots p_{i_l' j_l'}} \bigr]
\]
vanishes unless $\bm{i} \sqcup \bm{j}'$ is a permutation of $\bm{j} \sqcup \bm{i}'$.
In this case, the moment is equal to
\[
\alpha_1! \cdots \alpha_{n-1}! \frac{n^{\uparrow \alpha_n}}{n^{\uparrow m}n^{\uparrow l}},
\]
where
\[
\alpha_k := \mult_{k}(\bm{i}\sqcup \bm{j}') = \mult_{k}(\bm{j} \sqcup \bm{i}'),
\qquad k=1,\dots,n.
\]
\end{Theorem}

We postpone the proof of Theorem~\ref{thm:momentsP} to the next subsection.

\begin{Example}\label{ex:momentsP_examples}\quad
\begin{enumerate}\itemsep=0pt
\item Suppose $n>1$ and consider
\[
\mathbb{E} \bigl[|p_{11}|^4 |p_{nn}|^2\bigr]
=\mathbb{E}\bigl[(p_{11})^2 p_{nn} (\overline{p_{11}})^2 \overline{p_{nn}}\bigr].
\]
Apply Theorem~\ref{thm:momentsP} with
\[
\bm{i} =\bm{j}=\bm{i}'=\bm{j}'=(1,1,n).
\]
Then
\[
\mathbb{E} \bigl[|p_{11}|^4 |p_{nn}|^2\bigr] =4! \frac{n^{\uparrow 2}}{n^{\uparrow 3} n^{\uparrow 3}}
= \frac{24}{n(n+1) (n+2)^2}.
\]
\item Suppose $n>2$ and
consider $\mathbb{E}\bigl[(p_{12})^2 (p_{n1})^2(p_{nn})^3 (\overline{p_{n2}})^2 \bigr]$.
The data is
\[
\bm{i}=(1,1,n,n,n,n,n), \qquad \bm{j}=(2,2,1,1,n,n,n), \qquad \bm{i}'=(n,n), \qquad \bm{j}'=(2,2).
\]
Then
\[
\mathbb{E}\bigl[(p_{12})^2 (p_{n1})^2(p_{nn})^3 (\overline{p_{n2}})^2 \bigr]
= 2! \, 2! \, \frac{n^{\uparrow 5}}{n^{\uparrow 7} n^{\uparrow 2}}
= \frac{4}{n(n+1)(n+5)(n+6)}.
\]
\item Suppose $n>3$ and consider $\mathbb{E} \bigl[p_{12} p_{21} (p_{nn})^4 (\overline{p_{33}})^2
(\overline{p_{nn}})^3\bigr]$.
The data is
\[
\bm{i}=(1,2,n,n,n,n), \qquad \bm{j}=(2,1,n,n,n,n), \qquad \bm{i}'= \bm{j}'=(3,3,n,n,n).
\]
Then
\begin{align*}
\mathbb{E} \bigl[p_{12} p_{21} (p_{nn})^4 (\overline{p_{33}})^2
(\overline{p_{nn}})^3\bigr] &= 1! 1! 2! \frac{n^{\uparrow 7}}{n^{\uparrow 6} n^{\uparrow 5}}
 =
\frac{2(n+6)}{n(n+1)(n+2)(n+3)(n+4)}.
\end{align*}
\end{enumerate}
\end{Example}

\subsection{Proof of the theorem}

We rewrite Theorem~\ref{thm:momentsP} in an equivalent form.
For an $n \times n$ matrix $A=(a_{ij})_{1 \le i,j \le n}$ and for each $1 \le k \le n$,
we denote by $a_{k \bullet}$ the sum of all entries in the $k$-th row of $A$
and by $a_{\bullet k}$ the sum of all entries in the $k$-th column of $A$:
\[
a_{k \bullet} = \sum_{j=1}^n a_{kj}, \qquad a_{\bullet k} = \sum_{i=1}^n a_{ik}.
\]
Furthermore, we define
\[
\Sigma_A = \sum_{i=1}^n \sum_{j=1}^n a_{ij} = \sum_{k=1}^n a_{k \bullet} = \sum_{k=1}^n a_{\bullet k}.
\]
It is straightforward to verify that Theorem~\ref{thm:momentsP} is equivalent to Theorem~\ref{thm:momentsP2}.

\begin{Theorem} \label{thm:momentsP2}
Let $A=(a_{ij})_{1 \le i,j \le n}$ and $B=(b_{ij})_{1 \le i,j \le n}$ be two matrices whose entries are nonnegative integers.
Then the moment
\[
\mathbb{E} \Biggl[\prod_{i,j=1}^n (p_{ij})^{a_{ij}} (\overline{p_{ij}})^{b_{ij}} \Biggr]
\]
vanishes unless
\[
a_{\bullet k} + b_{k \bullet} = a_{k \bullet} + b_{\bullet k } =:\alpha_k
\]
for all $k=1,2,\dots, n$.
In this case, the moment is equal to
\[
 \alpha_1! \cdots \alpha_{n-1}!
\frac{n^{\uparrow \alpha_n}}{n^{\uparrow \Sigma_A} n^{\uparrow \Sigma_B}}.
\]
\end{Theorem}

\begin{proof}
By the definition \eqref{matrixP} of $(p_{ij})$, we have
\[
\prod_{i,j=1}^n (p_{ij})^{a_{ij}} (\overline{p_{ij}})^{b_{ij}}
= (-1)^c
\prod_{k=1}^{n-1}
\bigl( x_k^{a_{k \bullet}+b_{\bullet k}} \overline{x_k}^{a_{\bullet k} +b_{k \bullet}} \bigr)
\cdot (1-x_n)^{s} (1-\overline{x_n})^{t},
\]
where
\begin{align*}
&c= \sum_{j=1}^{n-1} (a_{nj}+b_{nj}) + \sum_{i=1}^{n-1} (a_{in}+b_{in}), \qquad
s = a_{n \bullet} -\sum_{j=1}^{n-1} b_{\bullet j}, \qquad
t = b_{n \bullet} -\sum_{j=1}^{n-1} a_{\bullet j}.
\end{align*}
Since $|x_n|<1$ almost surely,
the generalized binomial theorem implies that
\begin{align*}
&\mathbb{E} \Biggl[ \prod_{i,j=1}^n (p_{ij})^{a_{ij}} (\overline{p_{ij}})^{b_{ij}} \Biggr]
 = (-1)^{c}
 \sum_{m=0}^\infty\sum_{l=0}^\infty\frac{(-s)^{\uparrow m}(-t)^{\uparrow l}}{m! l! }
\mathbb{E} \Biggl[\prod_{k=1}^{n-1} \bigl( x_k^{a_{k \bullet}+b_{\bullet k}}
\overline{x_k}^{a_{\bullet k}+b_{k \bullet}} \bigr)
\cdot x_n^{m} \overline{x_n}^{l} \Biggr].
\end{align*}
Since we are integrating on the unit sphere, the interchange of the infinite sum and the expectation follows from Fubini's theorem.
Applying Lemma~\ref{lem:Wg_for_R} to each term, a term vanishes unless the equalities
$a_{\bullet k} + b_{k \bullet} = a_{k \bullet} + b_{\bullet k }$ hold for all $k=1,2,\dots,n-1$, and unless $m=l$ holds.
In this case, we also have $a_{\bullet n} + b_{n \bullet} = a_{n \bullet} + b_{\bullet n}$ and
$c \equiv 0 \pmod{2}$, which leads to the following equations:
\begin{equation} \label{eq:alpha_Sigma}
\begin{aligned}
&s+\alpha_1+ \cdots+ \alpha_{n-1}= \Sigma_A, \\
&t+\alpha_1+ \cdots+ \alpha_{n-1}= \Sigma_B, \\
&s+t+\alpha_1+ \cdots+ \alpha_{n-1}= \alpha_n.
\end{aligned}
\end{equation}
Furthermore, applying Lemma~\ref{lem:Wg_for_R} again, we obtain
\[
\mathbb{E} \Biggl[ \prod_{i,j=1}^n (p_{ij})^{a_{ij}} (\overline{p_{ij}})^{b_{ij}} \Biggr]
= \alpha_1 ! \cdots \alpha_{n-1}! \sum_{m=0}^\infty \frac{(-s)^{\uparrow m}(-t)^{\uparrow m}}{m! n^{\uparrow (\alpha_1+ \cdots+ \alpha_{n-1}+m)}}.
\]

We continue the calculations of the last series:
\begin{align*}
\sum_{m=0}^\infty \frac{(-s)^{\uparrow m}(-t)^{\uparrow m}}{m! n^{\uparrow (\alpha_1+ \cdots+ \alpha_{n-1}+m)}}
&=\frac{1}{n^{\uparrow (\alpha_1+\cdots+ \alpha_{n-1})}}
 \sum_{m=0}^\infty \frac{(-s)^{\uparrow m}(-t)^{\uparrow m}}{m! (n+\alpha_1+ \cdots+ \alpha_{n-1})^{\uparrow m}} \\
 &= \frac{1}{n^{\uparrow (\alpha_1+\cdots+ \alpha_{n-1})}} \cdot
 F(-s,-t;n+\alpha_1+ \cdots+ \alpha_{n-1};1),
\end{align*}
where $F(\alpha,\beta;\gamma;x)$ is the Gauss hypergeometric series.
The special value $F(\alpha,\beta;\gamma;1)$ makes sense if $\gamma-(\alpha+\beta)>0$, which is fulfilled in our setting.
Using the well-known formula $F(\alpha,\beta;\gamma;1)= \frac{\Gamma(\gamma) \Gamma(\gamma-\alpha-\beta)}{\Gamma(\gamma-\alpha)\Gamma(\gamma-\beta)}$ (see, e.g., \cite[Theorem 2.2.2]{AAR}),
where $\Gamma(x)$ is the gamma function,
we obtain
\[
F(-s,-t;n+\alpha_1+ \cdots+ \alpha_{n-1};1)
= \frac{\Gamma(n+\alpha_1+\cdots+\alpha_{n-1})\Gamma (n+\alpha_n)}{\Gamma(n+\Sigma_A) \Gamma(n+\Sigma_B)},
\]
where we used \eqref{eq:alpha_Sigma}.
The result follows by using $n^{\uparrow k}=\frac{\Gamma(n+k)}{\Gamma(n)}$.
\end{proof}

\subsection{Some examples for reflection matrices}

By utilizing Theorem~\ref{thm:momentsP} or Theorem~\ref{thm:momentsP2}, we can theoretically compute the moments for $(r_{ij})$.
However, when many diagonal elements $r_{ii}$ appear, many factors of the form $(1 - p_{ii})$ arise, which complicates the calculations using the theorems.

In general, the calculations tend to be complex; however, this subsection presents illustrative examples in cases where the number of diagonal elements is relatively small.
More intricate cases will be addressed in subsequent studies, as detailed in \cite{CardinCollinsMatsumoto}.

\begin{Proposition}\label{prop:moment_R_ex1}
Let $m$, $l$, $q$ be nonnegative integers, and consider four sequences:
\[
\bm{i}=(i_1,\dots, i_m), \qquad \bm{j} =(j_1,\dots, j_m), \qquad
\bm{i}'=(i_1',\dots, i_l'), \qquad \bm{j}' =(j_1',\dots, j_l')
\]
 of positive integers in $\{1,2,\dots,n\}.$
Then, for each $1 \le s \le n$ the moment
\[
\mathbb{E} \bigl[p_{i_1 j_1} \cdots p_{i_m j_m} \overline{p_{i_1' j_1'} \cdots p_{i_l' j_l'}} r_{ss}^q\bigr]
\]
 vanishes unless $\bm{i} \sqcup \bm{j}'$ is a permutation of $\bm{j} \sqcup \bm{i}'$.
In this case, the moment is equal to
\[
\alpha_1 ! \cdots \alpha_{n-1}! \, \frac{n^{\uparrow \alpha_n}}{n^{\uparrow (m+q)} n^{\uparrow l}}
\times
\begin{cases}
(n+m-\alpha_s-1)^{\uparrow q} & \text{if $s<n$}, \\
(m-\alpha_n)^{\uparrow q} & \text{if $s=n$},
\end{cases}
\]
where $\alpha_k:=\mult_{k}(\bm{i}\sqcup \bm{j}') = \mult_{k}(\bm{j}\sqcup\bm{i}')$
for all $k=1,2,\dots, n$.
Note that the last factor may be zero.
\end{Proposition}

\begin{proof}
Applying the binomial theorem to $r_{ss}^q=(1-p_{ss})^q$, we have
\begin{align}
&\mathbb{E} \bigl[p_{i_1 j_1} \cdots p_{i_m j_m} \overline{p_{i_1' j_1'} \cdots p_{i_l' j_l'}} r_{ss}^q\bigr] = \sum_{k=0}^q \frac{(-q)^{\uparrow k}}{k!} \mathbb{E} \bigl[p_{i_1 j_1} \cdots p_{i_m j_m} \overline{p_{i_1' j_1'} \cdots p_{i_l' j_l'}}p_{ss}^k \bigr]. \label{eq:1_diagonal_R}
\end{align}
First, we deal with the case $s <n$. Applying Theorem~\ref{thm:momentsP} to each term, we have
\begin{align*}
\begin{split}
\eqref{eq:1_diagonal_R} &=
 \sum_{k=0}^q \frac{(-q)^{\uparrow k}}{k!} \alpha_1! \cdots \alpha_{n-1}! \cdot \frac{(\alpha_s+ k)!}{\alpha_s!}
 \cdot \frac{n^{\uparrow \alpha_n}}{n^{\uparrow (m+k)} n^{\uparrow l}} \\
 &= \alpha_1! \cdots \alpha_{n-1}! \cdot \frac{n^{\uparrow \alpha_n}}{n^{\uparrow m} n^{\uparrow l}} \sum_{k=0}^q \frac{(-q)^{\uparrow k} (\alpha_s+1)^{\uparrow k}}{k! (n+m)^{\uparrow k}}.
 \end{split}
\end{align*}
Using a similar method to that in the proof of Theorem~\ref{thm:momentsP}, we recognize the last sum as a~Gauss hypergeometric series and obtain
\begin{align*}
 &= \alpha_1!\, \cdots \, \alpha_{n-1}! \cdot \frac{n^{\uparrow \alpha_n}}{n^{\uparrow m} n^{\uparrow l}} \cdot \frac{(n+m-\alpha_s-1)^{\uparrow q}}{(n+m)^{\uparrow q}} \\
 &= \alpha_1!\, \cdots \, \alpha_{n-1}! \cdot \frac{n^{\uparrow \alpha_n}}{n^{\uparrow (m+q)} n^{\uparrow l}} \cdot (n+m-\alpha_s-1)^{\uparrow q}.
\end{align*}
This concludes the proof for the case $s < n$.

Next, we deal with the case $s =n$. A similar approach gives
\begin{align*}
\eqref{eq:1_diagonal_R} &= \sum_{k=0}^q \frac{(-q)^{\uparrow k}}{k!} \alpha_1! \cdots \alpha_{n-1}! \frac{n^{\uparrow (\alpha_n +k)}}{n^{\uparrow (m+k)} n^{\uparrow l}}
 = \alpha_1! \cdots \alpha_{n-1}! \frac{n^{\uparrow \alpha_n}}{n^{\uparrow m} n^{\uparrow l}}\sum_{k=0}^q \frac{(-q)^{\uparrow k} (n+\alpha_n)^{\uparrow k}}{k! (n+m)^{\uparrow k}} \\
&=\alpha_1! \cdots \alpha_{n-1}! \frac{n^{\uparrow \alpha_n}}{n^{\uparrow m} n^{\uparrow l}} \cdot \frac{(m-\alpha_n)^{\uparrow q}}{(n+m)^{\uparrow q}}.
\end{align*}
This concludes the proof for the case when $s = n$.
\end{proof}

\begin{Example}\label{ex:rss_q}
Let $q$ be a positive integer. Then
\[
\mathbb{E} [ r_{ss}^q ]= \begin{cases}
\dfrac{n-1}{n+q-1} & \text{if $1 \le s<n$}, \\
0 & \text{if $s=n$}.
\end{cases}
\]
\end{Example}

\begin{Example}\label{ex:mixed_r}
If $n > 2$, then
\begin{align*}
\mathbb{E} \bigl[r_{12}^2 r_{n1}^2 \overline{r_{n2}}^2 r_{22}^3\bigr]
= \mathbb{E} \bigl[p_{12}^2 p_{n1}^2 \overline{p_{n2}}^2 r_{22}^3\bigr]
&=\frac{4}{n(n+4)(n+5)(n+6)}, \\
\mathbb{E} \bigl[r_{12}^2 r_{n1}^2 \overline{r_{n2}}^2 r_{nn}^3\bigr]
=\mathbb{E} \bigl[p_{12}^2 p_{n1}^2 \overline{p_{n2}}^2 r_{nn}^3\bigr]
&= \frac{96}{n^{\uparrow 7}}.
\end{align*}
\end{Example}

\begin{Proposition} \label{prop:two_r}
Let $m$, $l$ be nonnegative integers, and consider four sequences
\[
\bm{i}=(i_1,\dots, i_m), \qquad \bm{j} =(j_1,\dots, j_m), \qquad
\bm{i}'=(i_1',\dots, i_l'), \qquad \bm{j}' =(j_1',\dots, j_l')
\]
 of positive integers in $\{1,2,\dots,n\}.$
Then, for each $1 \le s \le n$, the moment
\[
\mathbb{E} \bigl[p_{i_1 j_1} \cdots p_{i_m j_m} \overline{p_{i_1' j_1'} \cdots p_{i_l' j_l'}} |r_{ss}|^2 \bigr]
\]
 vanishes unless $\bm{i} \sqcup \bm{j}'$ is a permutation of $\bm{j} \sqcup \bm{i}'$.
In this case, the moment is equal to
\begin{align*}
&\alpha_1 ! \cdots \alpha_{n-1}! \frac{n^{\uparrow \alpha_n}}{n^{\uparrow (m+1)} n^{\uparrow (l+1)}} \\
& \qquad \times
\begin{cases}
 ( (n+m)(n+l) -(\alpha_s+1)(2n+m+l) +(\alpha_s+1)(\alpha_s+2) ) & \text{if $s<n$},
\\
 ( n +(m-\alpha_n)(l-\alpha_n) + \alpha_n ) & \text{if $s=n$},
\end{cases}
\end{align*}
where $\alpha_k:=\mult_{k}(\bm{i}\sqcup \bm{j}') = \mult_{k}(\bm{j}\sqcup\bm{i}')$
for all $k=1,2,\dots, n$.
\end{Proposition}

\begin{proof}
Expand $|r_{ss}|^2 = 1- p_{ss} -\overline{p_{ss}} +p_{ss}\overline{p_{ss}}$ and apply Theorem~\ref{thm:momentsP}.
The detailed calculations are omitted.
\end{proof}

\begin{Example} \label{example:two_r}
 Let $n>1$. From Proposition~\ref{prop:two_r}, we have
 \[
 \mathbb{E}\bigl[|r_{nn}|^2\bigr] = \frac{1}{n}, \qquad
 \mathbb{E}\bigl[p_{11} |r_{nn}|^2\bigr] =\frac{1}{n(n+1)}, \qquad
 \mathbb{E}\bigl[p_{11} \overline{p_{11}}|r_{nn}|^2\bigr] = \frac{2}{n^2(n+1)}.
 \]
 Using these results, we obtain the following:
 \begin{align*}
 \mathbb{E}\bigl[|r_{11} r_{nn}|^2\bigr] = \mathbb{E}\bigl[(1-p_{11})(1-\overline{p_{11}}) |r_{nn}|^2\bigr]
 = \frac{1}{n} -2 \cdot \frac{1}{n(n+1)}+ \frac{2}{n^2(n+1)} = \frac{n^2-n+2}{n^2(n+1)}.
 \end{align*}
\end{Example}

\begin{Proposition}
Let $m$ be a positive integer, and consider two sequences
\[
\bm{i}=(i_1,\dots, i_m), \quad \bm{j} =(j_1,\dots, j_m)
\]
 of positive integers in $\{1,2,\dots,n\}.$
Assume that $i_1,\dots,i_m$ are distinct
and that none of them equals $n$.
Then the moment
\[
\mathbb{E} \left[ r_{i_1 j_1} r_{i_2 j_2} \cdots r_{i_m j_m} \right]
\]
vanishes unless $\bm{j}$ is a permutation of $\bm{i}$.
In this case, the moment is equal to
\begin{equation} \label{eq:ascention_proto}
\sum_{t=0}^f (-1)^{m-t} \binom{f}{t} \frac{1}{n^{\uparrow (m-t)}},
\end{equation}
where $f=f(\bm{i},\bm{j})=|\{h \in \{1,2,\dots,m\} \mid i_h=j_h \}|$.
\end{Proposition}

\begin{proof}
We may assume that $i_h =j_h$ for $1 \le h \le f$, and $i_h \not=j_h$ for $f+1 \le h \le m$.
We define
\begin{gather*}
\bm{f}:= (i_1,i_2, \dots, i_f), \qquad
\bm{i}_0:= (i_{f+1}, i_{f+2}, \dots, i_{m}), \qquad
\bm{j}_0:= (j_{f+1}, j_{f+2}, \dots, j_{m}).
\end{gather*}
As the components of $\bm{f}$ are distinct,
the sequence $\bm{f}$ can be identified with the set $\{i_1,i_2,\dots,i_f\}$.
We then have
\begin{align*}
 \mathbb{E} \Biggl[ r_{i_1 j_1} r_{i_2 j_2} \cdots r_{i_m j_m} \Biggr]
 &= \mathbb{E} \Biggl[\prod_{i \in \bm{f}} (1- p_{i, i}) \cdot \prod_{h=f+1}^m (-p_{i_h, j_h}) \Biggr] \\
 &= \sum_{\bm{k}: \bm{k} \subset \bm{f}} \mathbb{E} \Biggl[\prod_{i \in \bm{k}} (-p_{i,i}) \cdot\prod_{h=f+1}^m (-p_{i_h,j_h}) \Biggr],
\end{align*}
where $\bm{k}$ runs over all subsets of the set $\bm{f}$.
Applying Theorem~\ref{thm:momentsP} to each term,
we find that a~term corresponding to $\bm{k}$ is nonzero only if $\bm{k} \sqcup \bm{j}_0$ is a permutation of $\bm{k} \sqcup \bm{i}_0$. In other words, $\bm{j}$ must be a permutation of $\bm{i}$.
Furthermore, the calculation continues as follows:
\[
 \mathbb{E} \left[ r_{i_1 j_1} r_{i_2 j_2} \cdots r_{i_m j_m} \right]
 = \sum_{s=0}^f \binom{f}{s} (-1)^{s+m-f}
 \frac{1}{n^{\uparrow (s+m-f)}}.
\]
The desired form is obtained by changing $s \mapsto t =f-s$.
\end{proof}

The following example plays an important theoretical role in Section~\ref{section:Wg_recursive}.

\begin{Example} \label{ex:permutation_r}
For $k<n$ and $\sigma \in S_k$, we have
 \[
 \mathbb{E}[r_{1\sigma(1)} \cdots r_{k \sigma(k)}]= \sum_{t=0}^{\fix{\sigma}} (-1)^{k-t} \binom{\fix{\sigma}}{t} \frac{1}{n^{\uparrow (k-t)}},
 \]
 where $\fix{\sigma}$ is the number of fixed points in $\sigma$:
 \[
 \fix{\sigma} =|\{i \in \{1,2,\dots,k\} \mid \sigma(i)=i\}|.
 \]
\end{Example}

\begin{Remark}
We can express the function given in \eqref{eq:ascention_proto}
using Kummer's confluent hypergeometric function
\[
M(a,b,z)= \sum_{k=0}^\infty \frac{a^{\uparrow k}}{b^{\uparrow k}} \frac{z^k}{k!}.
\]
Indeed,
\[
\sum_{t=0}^f (-1)^{m-t} \binom{f}{t} \frac{1}{n^{\uparrow (m-t)}}
= \frac{(-1)^{m-f}}{n^{\uparrow (m-f)}} M(-f,n+m-f,1).
\]
\end{Remark}

\subsection{Calculation of Haar-distributed matrices from reflection matrices}\label{sec:compute-integral}

The following lemma is a reformulation of Lemma~\ref{lem:BNN}.

\begin{Lemma}
\label{lem:BNN_g}
Let $X_1,X_2, \dots, X_n$ be independent random vectors, where
each $X_k$ is uniformly distributed on the unit sphere of $\mathbb{C}^k$.
For each $1 \le k \le n$, we denote $R_k$ as the corresponding reflection matrix of size $k$, defined similarly to \eqref{matrixR},
and define an $n \times n$ matrix $\tilde{R}_k = R_k \oplus I_{n-k}$.
Then, a Haar-distributed unitary matrix $U$ in $\UU(n)$ has the same distribution as
the product $\tilde{R}_n \tilde{R}_{n-1} \cdots \tilde{R}_1$.
In particular, the last $k$ columns of $U$ match those of $\tilde{R}_n \tilde{R}_{n-1} \cdots \tilde{R}_{n-k+1}$.
\end{Lemma}

Using this lemma, one can calculate moments of polynomial functions on the unitary group $\UU(n)$ by a method different from the Weingarten calculus, employing Theorem~\ref{thm:momentsP}.
In this subsection, we provide some examples of such calculations, which indicate that the approach described in this paper allows one to recover Weingarten functions.

Let $n \ge 2$, and let $U=(u_{ij})$ be a Haar-distributed unitary matrix on $\UU(n)$.
From Lem\-ma~\ref{lem:BNN_g}, the last two columns of $U$ are expressed as follows:
\[
u_{i,n-1} = \sum_{j=1}^{n-1} r_{i,j} s_{j,n-1}, \qquad u_{i,n} =r_{i,n},
\qquad 1 \le i \le n,
\]
where $r_{i,j}$ are matrix elements of $R_n$ and $s_{i,j}$ are those of $R_{n-1}$.
Let us calculate the following moment.
\[
\mathbb{E} [ u_{n-1,n-1} u_{n,n} \overline{u_{n-1,n} u_{n, n-1}} ].
\]
By applying the Weingarten calculus \eqref{eq:Wg_F}, one immediately obtains
\[\mathbb{E} [ u_{n-1,n-1} u_{n,n} \overline{u_{n-1,n} u_{n, n-1}} ]=\Wg_{2,n}( (1 \ 2) ) = - \frac{1}{n(n+1)(n-1)},\]
but here, we shall deliberately take a longer route to calculate it.
Noting that $(r_{ij})$ and $(s_{ij})$ are independent, it follows that
\begin{align*}
 \mathbb{E} [ u_{n-1,n-1} u_{n,n} \overline{u_{n-1,n} u_{n, n-1}} ]
 =\sum_{j=1}^{n-1} \sum_{l=1}^{n-1} \mathbb{E} [r_{n-1,j} r_{n,n} \overline{r_{n-1,n} r_{n,l}}]
\mathbb{E}[s_{j,n-1} \overline{s_{l,n-1}}].
\end{align*}
Example~\ref{example:two_r} or Proposition~\ref{prop:moment_R_ex1} give that
\begin{gather*}
\mathbb{E}[s_{j,n-1} \overline{s_{l,n-1}}] = \frac{\delta_{j,l}}{n-1} \qquad \text{for $j,l \le n-1$}, \\
\mathbb{E} [r_{n-1,j} r_{n,n} \overline{r_{n-1,n} r_{n,j}}] =0 \qquad \text{for $j<n-1$},
\end{gather*}
and
\begin{align*}
\mathbb{E} [r_{n-1,n-1} r_{n,n} \overline{r_{n-1,n} r_{n,n-1}}]
 = \mathbb{E} [ (1-p_{n-1,n-1}) \overline{p_{n-1,n}p_{n,n-1}} r_{n,n}]
 = -\frac{1}{n(n+1)}.
\end{align*}
Thus, we have obtained
\begin{align*}
 \mathbb{E} [ u_{n-1,n-1} u_{n,n} \overline{u_{n-1,n} u_{n, n-1}} ]
 &=\mathbb{E} [r_{n-1,n-1} r_{n,n} \overline{r_{n-1,n} r_{n,n-1}}]
\mathbb{E}[s_{n-1,n-1} \overline{s_{n-1,n-1}}] \\
 &= -\frac{1}{n(n+1)(n-1)}
\end{align*}
as desired.
Let us perform a similar calculation for $\mathbb{E} \left[ u_{n-1,n-1} u_{n,n} \overline{u_{n-1,n-1} u_{n, n}} \right]$, which equals $\Wg_{2,n}(e_2)=\frac{1}{(n-1)(n+1)}$ by Weingarten calculus.
We see that
\begin{align*}
\mathbb{E} [ u_{n-1,n-1} u_{n,n} \overline{u_{n-1,n-1} u_{n, n}} ]
&= \sum_{j=1}^{n-1} \sum_{l=1}^{n-1} \mathbb{E} [r_{n-1,j} r_{n,n} \overline{r_{n-1,l} r_{n,n}}]
\mathbb{E}[s_{j,n-1} \overline{s_{l,n-1}}] \\
&= \frac{1}{n-1} \sum_{j=1}^{n-1} \mathbb{E} [r_{n-1,j} r_{n,n} \overline{r_{n-1,j} r_{n,n}}].
\end{align*}
Using Proposition~\ref{prop:two_r} (see also Example~\ref{example:two_r}), we have:
\[
\mathbb{E} [r_{n-1,j} r_{n,n} \overline{r_{n-1,j} r_{n,n}}] =
\begin{cases}
 \dfrac{1}{n^2(n+1)} & \text{if $j <n-1$}, \vspace{1mm}\\
 \dfrac{n^2-n+2}{n^2(n+1)} & \text{if $j =n-1$}.
\end{cases}
\]
This gives
\begin{align*}
&\mathbb{E} [ u_{n-1,n-1} u_{n,n} \overline{u_{n-1,n-1} u_{n, n}} ]\\
& \qquad
= \frac{1}{n-1} \left( (n-2) \cdot \frac{1}{n^2(n+1)} + \frac{n^2-n+2}{n^2(n+1)}
\right) = \frac{1}{(n-1)(n+1)},
\end{align*}
as desired.

This approach works for any integral of a monomial in $u_{ij}$ and their conjugates,
provided it has a single matching.
Such an integral results in a Weingarten function, and this provides a~method to recompute the Weingarten function.

More importantly, no restriction relating the degree of the monomial to the dimension is required.
Indeed, expanding each entry $u_{ij}$ via Lemma~\ref{lem:BNN_g} and writing each reflection entry as $r_{ij}= \delta_{ij}-p_{ij}$, the independence of the reflections factorizes the expectation into a finite sum of products of moments covered by Theorem~\ref{thm:momentsP}, which holds without any assumption on the degree.
This yields an algorithm that terminates in finitely many steps for every polynomial integral on $\UU(n)$, including integrals of degree larger than $n$.
We illustrate this on a family of integrals of arbitrarily high degree in dimension $2$.

\begin{Example} \label{ex:U2_high_degree}
Let $n=2$ and let $k$ be an arbitrary positive integer.
By Lemma~\ref{lem:BNN_g}, we may write $u_{11}=r_{11}\, x^{(1)}$, where $x^{(1)}$ is uniformly distributed on the unit circle of $\mathbb{C}$ and independent of $R_2=(r_{ij})$.
Since $\bigl|x^{(1)}\bigr|=1$ and $r_{11}=1-p_{11}$, the binomial theorem yields
\[
\mathbb{E}\bigl[|u_{11}|^{2k}\bigr] = \mathbb{E}\bigl[|1-p_{11}|^{2k}\bigr]
= \sum_{a=0}^{k} \sum_{b=0}^{k} (-1)^{a+b} \binom{k}{a} \binom{k}{b} \mathbb{E}\bigl[(p_{11})^a (\overline{p_{11}})^b\bigr].
\]
Theorem~\ref{thm:momentsP}, applied with $\bm{i}=\bm{j}=(1,\dots,1)$ ($a$ times) and $\bm{i}'=\bm{j}'=(1,\dots,1)$ ($b$ times), gives
\[
\mathbb{E}\left[(p_{11})^a (\overline{p_{11}})^b\right] = \frac{(a+b)!}{2^{\uparrow a}\, 2^{\uparrow b}} = \frac{(a+b)!}{(a+1)!\,(b+1)!},
\]
so that
\[
\mathbb{E}\left[|u_{11}|^{2k}\right] = \sum_{a,b=0}^{k} (-1)^{a+b} \binom{k}{a}\binom{k}{b} \frac{(a+b)!}{(a+1)!\,(b+1)!} = \frac{1}{k+1}.
\]
The evaluation of the double sum agrees with the classical fact that $|u_{11}|^2$ is uniformly distributed on $[0,1]$ when $n=2$.
The point here is that our method produces the finite closed sum above in a purely algorithmic way, for a monomial whose degree exceeds the dimension for every $k \ge 2$.
\end{Example}

It remains to identify the underpinning algebraic and combinatorial phenomena that allow one to systematically rederive the known results on Weingarten functions, which we leave as an open problem.
For a related result, we also refer to Proposition~\ref{prop:Wg_asasas}.

\section{A new recursive formula for Weingarten functions} \label{section:Wg_recursive}

\subsection{A recursive formula}

In this section, we rely on the achievements of the previous paragraphs to uncover a new relation between the Weingarten functions
$\Wg_{k,n}$ and $\Wg_{k,n-1}$.
Let $\mathbb{C}[S_k]$ denote the group algebra of the symmetric group,
which consists of all complex-valued functions on $S_k$.
The group algebra is equipped with a convolution product defined as follows:
\[
(f_1*f_2)(\pi)= \sum_{\sigma \in S_k} f_1(\sigma) f_2(\sigma^{-1} \pi), \qquad \pi \in S_k.
\]
Here, $f_1$ and $f_2$ are elements of $\mathbb{C}[S_k]$.
If $f_1$ and $f_2$ are class functions on $S_k$, which means that they are constant on each conjugacy class of~$S_k$,
then the convolution can also be written as
\[
(f_1*f_2)(\pi)= \sum_{\sigma \in S_k} f_1(\sigma) f_2(\sigma \pi),
\]
and moreover, the convolution is commutative; that is, $f_1*f_2=f_2*f_1$.
Here, the commutativity is shown by the equality
\[
(f_2*f_1)(\pi) =\sum_{\sigma \in S_k} f_2(\sigma) f_1\bigl(\sigma^{-1} \pi\bigr)
= \sum_{\sigma \in S_k} f_2(\sigma) f_1\bigl(\pi \sigma^{-1}\bigr) =(f_1*f_2)(\pi).
\]
We will use this commutativity property without further notice.

We introduce a new element of $\mathbb{C}[S_k]$, which has already appeared in Example~\ref{ex:permutation_r}.

\begin{Definition} \label{def:function_a}
Let $k$ and $n$ be positive integers.
We define the \emph{ascension function} $\Raise_{k,n}$ in~$\mathbb{C}[S_k]$ by
 \[
\Raise_{k,n}(\sigma)= \sum_{t=0}^{\fix{\sigma}} (-1)^{k-t} \binom{\fix{\sigma}}{t} \frac{1}{n^{\uparrow (k-t)}},
\qquad \sigma \in S_k,
 \]
 where $\fix{\sigma}$ is the number of fixed points in $\sigma$:
 \[
 \fix{\sigma} =|\{i \in \{1,2,\dots,k\} \mid \sigma(i)=i\}|.
 \]
\end{Definition}

\begin{Example}
For the identity permutation $e_k$ in $S_k$, we have
\[
\Raise_{k,n}(e_k)
= 1-\frac{k}{n} + \frac{\binom{k}{2}}{n(n+1)} - \cdots +(-1)^k \frac{1}{n(n+1) \cdots(n+k-1)}.
\]
If $\sigma \in S_k$ has no fixed points, then
\[
\Raise_{k,n}(\sigma)= \frac{(-1)^k}{n(n+1) \cdots (n+k-1)}.
\]
\end{Example}

Recall the Weingarten function which already appeared in the introduction. Suppose that $k \le n$.
The Weingarten function for the unitary group $\UU(n)$ is given by
\[
\Wg_{k,n}(\pi) = \mathbb{E}\bigl[ u_{11} u_{22} \cdots u_{kk}
\overline{u_{1\pi(1)} u_{2\pi(2)} \cdots u_{k\pi(k)}}\bigr],
\]
for each permutation $\pi\in S_k$,
where $U=(u_{ij})$ is a Haar-distributed unitary matrix in $\UU(n)$.
The name `ascension function' for $\Raise_{k,n}$ originates from the following theorem.

\begin{Theorem} \label{thm:recursive}
Let $k$ and $n$ be positive integers and suppose $k+1 \le n$.
Then we have the convolution identity
\[
\Wg_{k,n} =\Raise_{k,n} * \Wg_{k,n-1}.
\]
\end{Theorem}

\begin{Remark}
Theorem~\ref{thm:recursive} does not hold without the assumption $k+1 \le n$. In the case where $k+1 > n$, the Weingarten function $\Wg_{k,n-1}$ on the right-hand side is not well-defined. 
Although the Weingarten functions can still be extended in that case, the above formula does not hold.
We will elaborate on these points in Section~\ref{subsec:low_dim}.
\end{Remark}

The following lemma will be used for the proof of Theorem~\ref{thm:recursive}.

\begin{Lemma} \label{lem:expression_A}
Let $X=\trans{(x_1,\dots,x_n)}$ be a uniform vector on the unit sphere of $\mathbb{C}^n$.
Then, for $k \le n$ and any $\sigma \in S_k$, we have
\[
\Raise_{k,n}(\sigma)=\mathbb{E} \Biggl[
\prod_{i=1}^k \bigl( \delta_{i,\sigma(i)} - x_i \overline{x_{\sigma(i)}} \bigr) \Biggr].
\]
\end{Lemma}

\begin{proof}
We denote the set of all fixed points of $\sigma$ by
$\Fix (\sigma)= \{i \in\{1,2,\dots,k\} \mid \sigma(i)=i\}$.
By expanding the product, we have
\begin{align*}
 \mathbb{E} \Biggl[
\prod_{i=1}^k \bigl( \delta_{i,\sigma(i)} - x_i \overline{x_{\sigma(i)}} \bigl) \Biggr] &
 = (-1)^k \mathbb{E} \Biggl[ \prod_{i \in \Fix(\sigma)} \bigl(|x_i|^2-1\bigr) \cdot
 \prod_{j \in \Fix(\sigma)^c} x_j \overline{x_{\sigma(j)}} \Biggr] \\
 &= \sum_{s=0}^{\fix{\sigma}} (-1)^{k+\fix{\sigma}-s} \!\!\sum_{ \{i_1< \cdots <i_{s} \} \subset \Fix(\sigma)}\!\!
 \mathbb{E} \Biggl[ |x_{i_1}|^2 \cdots |x_{i_{s}}|^2 \prod_{j \in \Fix(\sigma)^c}|x_j|^2\Biggr],
\end{align*}
where $\Fix(\sigma)^c :=\{1,2,\dots,k \} \setminus \Fix(\sigma)$ is invariant under $\sigma$.
By Lemma~\ref{lem:Wg_for_R}, each expectation appearing in the sum equals $1/n^{\uparrow (s+k-\fix{\sigma})}$.
Therefore, we obtain
\begin{align*}
 \mathbb{E} \Biggl[
\prod_{i=1}^k \bigl( \delta_{i,\sigma(i)} - x_i \overline{x_{\sigma(i)}} \bigl) \Biggr]
&= \sum_{s=0}^{\fix{\sigma}} (-1)^{k+\fix{\sigma}-s} \binom{\fix{\sigma}}{s} \frac{1}{n^{\uparrow (s+k-\fix{\sigma})}} \\
&= \sum_{t=0}^{\fix{\sigma}} (-1)^{k-t} \binom{\fix{\sigma}}{t} \frac{1}{n^{\uparrow (k-t)}}.\tag*{\qed}
\end{align*}\renewcommand{\qed}{}
\end{proof}

\begin{proof}[Proof of Theorem~\ref{thm:recursive}]
 Let the matrices $U=(u_{ij})$, $R=(r_{ij})$, and $V=(v_{ij})$ be as stated in Lemma~\ref{lem:BNN}.
 Then, for $i,j <n$, the $(i,j)$-th element of $U$ is given by
 \[
 u_{ij} = \sum_{p=1}^{n-1} r_{ip} v_{p j}.
 \]
 Let $\pi \in S_k$. We consider
 \[
 \Wg_{k,n}(\pi) = \mathbb{E}\bigl[ u_{1,1} u_{2,2} \cdots u_{k,k}
 \overline{u_{\pi(1),1} u_{\pi(2),2} \cdots u_{\pi(k),k}}\bigr].
 \]
 Since we assume $k+1 \le n$, the elements of $U$ in the $n$-th row or $n$-th column do not appear in this equation.
 By the independence of $R$ and $V$, we have
 \begin{align*}
 \Wg_{k,n}(\pi)
={}& \sum_{p_1=1}^{n-1} \cdots \sum_{p_k=1}^{n-1} \sum_{q_1=1}^{n-1} \cdots \sum_{q_k=1}^{n-1}
 \mathbb{E}\bigl[ r_{1, p_1} \cdots r_{k, p_k} \overline{r_{\pi(1),q_1} \cdots r_{\pi(k), q_k}}\bigr] \\
 & \hphantom{\sum_{p_1=1}^{n-1} \cdots \sum_{p_k=1}^{n-1} \sum_{q_1=1}^{n-1} \cdots \sum_{q_k=1}^{n-1}}{}
 \times \mathbb{E} \bigl[ v_{p_1, 1} \cdots v_{p_k, k} \overline{ v_{q_1, 1} \cdots v_{q_k, k} } \bigr].
 \end{align*}
 Next, we apply the Weingarten formula \eqref{eq:Wg_F} for the unitary group $\UU(n-1)$ to the last factor:
 \begin{align*}
 \mathbb{E} \bigl[ v_{p_1, 1} \cdots v_{p_k, k} \overline{ v_{q_1, 1} \cdots v_{q_k, k} } \bigr]
 = \sum_{\tau \in S_k} \delta(p_{\tau^{-1}(1)},q_1) \cdots \delta(p_{\tau^{-1}(k)},q_k) \Wg_{k,n-1}(\tau).
 \end{align*}
 Substituting this into the preceding equation, we obtain
 \begin{align*}
 \Wg_{k,n}(\pi)
 = \sum_{\tau \in S_k} \Wg_{k,n-1}(\tau) \sum_{p_1=1}^{n-1} \cdots \sum_{p_k=1}^{n-1}
 \mathbb{E} \bigl[ r_{1,p_1} \cdots r_{k,p_k}
 \overline{r_{\pi\tau(1),p_1} \cdots r_{\pi\tau(k),p_k}} \bigr],
 \end{align*}
which can be written as $\Wg_{k,n} (\pi)= (\widetilde{\Raise} * \Wg_{k,n-1})(\pi)$ with
\[
\widetilde{\Raise}(\sigma) := \sum_{p_1=1}^{n-1} \cdots \sum_{p_k=1}^{n-1}
\mathbb{E} \bigl[ r_{1,p_1} \cdots r_{k,p_k}
\overline{r_{\sigma(1),p_1} \cdots r_{\sigma(k),p_{k}}}\bigr], \qquad \sigma \in S_k.
\]

The final step in the proof is to show that $\widetilde{\Raise}(\sigma)$ equals $\Raise_{k,n}(\sigma)$, as defined in Definition~\ref{def:function_a}.
Starting from the expression for $\widetilde{\Raise}(\sigma)$, we have
\[
\widetilde{\Raise}(\sigma)
= \mathbb{E} \Biggl[
\Biggl(\sum_{p_1=1}^{n-1} r_{1,p_1}\overline{r_{\sigma(1),p_1}} \Biggr)
\cdots
\Biggl(\sum_{p_k=1}^{n-1} r_{k,p_k}\overline{r_{\sigma(k),p_k}} \Biggr)
\Biggr].
\]
Since $R=(r_{ij})$ is an $n \times n$ unitary matrix,
the sums can be simplified as follows:
\[
\sum_{p=1}^{n-1} r_{i,p}\overline{r_{\sigma(i),p}} =
\sum_{p=1}^{n} r_{i,p}\overline{r_{\sigma(i),p}} -r_{i,n} \overline{r_{\sigma(i),n}}
= \delta_{i,\sigma(i)}-r_{i,n} \overline{r_{\sigma(i),n}}
\]
for all $1 \le i \le k$.
From \eqref{matrixR}, we know that $r_{i,n}=x_i$.
Substituting this, we obtain
\[
\widetilde{\Raise}(\sigma)= \mathbb{E} \Biggl[ \prod_{i=1}^k \big(\delta_{i,\sigma(i)}-x_{i} \overline{x_{\sigma(i)}}\big) \Biggr].
\]
Together with Lemma~\ref{lem:expression_A}, this shows $\widetilde{\Raise}(\sigma)=\Raise_{k,n}(\sigma)$.
This completes the proof of Theorem~\ref{thm:recursive}.
\end{proof}

\subsection{Some properties for ascension functions}

In this subsection, we investigate some properties of the function $\Raise_{k,n}\colon S_k \to \mathbb{C}$.
To do this, we review some fundamental properties of the Weingarten function;
refer to \cite{CMN} and its references for details.

Consider the function $\G_{k,n}\colon S_k \to \mathbb{C}$ defined by
\[
\G_{k,n}(\pi)= n^{\kappa(\pi)}, \qquad \pi \in S_k,
\]
where $\kappa(\pi)$ denotes the number of cycles in the cycle decomposition of $\pi$.
This function is invertible if and only if $k \le n$, and in this case,
its inverse function is the Weingarten function~$\Wg_{k,n}$:
\[
\G_{k,n} * \Wg_{k,n} = \Wg_{k,n} *\G_{k,n} =\delta_{e_k}.
\]
Here, $\delta_{e_k}\colon S_k \to \{0,1\}$ is the Dirac function at the identity permutation $e_k$:
\[
\delta_{e_k}( \sigma)= \begin{cases}
1 & \text{if $\sigma=e_k$}, \\
0 & \text{otherwise}.
\end{cases}
\]
From Theorem~\ref{thm:recursive}, it immediately follows that
\begin{equation} \label{eq:as_g_Wg}
 \Raise_{k,n} = \G_{k,n-1} * \Wg_{k,n} = \Wg_{k,n} * \G_{k,n-1}
\end{equation}
provided that $k +1 \le n$.

Next, we consider character expansions for these functions.
A partition of a positive integer~$k$ is a weakly decreasing sequence
$\lambda=(\lambda_1,\lambda_2,\dots,\lambda_l)$ of positive integers
such that the sum of all parts $\lambda_i$ is equal to~$k$.
In this case, we write $\lambda \vdash k$
and $\ell(\lambda)=l$.
For a partition $\lambda \vdash k$ and a~complex number~$z$, we define
\[
(z\uparrow \lambda)= \prod_{i=1}^{\ell(\lambda)} \prod_{j=1}^{\lambda_i} (z+j-i).
\]
Denote by $\chi^\lambda$ the irreducible character of $S_k$ corresponding to $\lambda$.
Moreover, we define $f^\lambda :=\chi^\lambda(e_k)$,
the value of $\chi^\lambda$ at the identity permutation.
Using these notations, we have the following expansions:
\begin{equation} \label{eq:Wg_char}
 \G_{k,n} = \frac{1}{k!} \sum_{\lambda \vdash k} f^\lambda (n \uparrow \lambda) \chi^\lambda,
\qquad
\Wg_{k,n} = \frac{1}{k!} \sum_{\lambda \vdash k} \frac{f^\lambda}{(n \uparrow \lambda)} \chi^\lambda.
\end{equation}
Note that we need to assume $k \le n$ in the latter equation
to ensure that the factor $(n \uparrow \lambda)$ in all terms is not equal to zero.
From these formulas, we can easily obtain the character expansion of the ascension function by using the orthogonality of irreducible characters:
\[
\chi^\lambda * \chi^\mu = \delta_{\lambda,\mu} \frac{k!}{f^\lambda} \chi^\lambda, \qquad \lambda,\mu \vdash k.
\]
Indeed, \eqref{eq:as_g_Wg} implies that
\begin{equation*} 
 \Raise_{k,n} = \frac{1}{k!} \sum_{\lambda \vdash k} f^\lambda
 \Biggl( \prod_{i=1}^{d(\lambda)} \frac{n-1-(\lambda_i'-i)}{ n+(\lambda_i-i)} \Biggr) \chi^\lambda,
\end{equation*}
where $(\lambda'_1,\lambda_2',\dots)$ represents the conjugate partition of $\lambda$,
obtained by transposing the partition, viewed as a Young diagram, along its diagonal.
Additionally, $d(\lambda)=|\{i \ge 1 \mid \lambda_i \ge i\}|$ denotes the length of the diagonal in the Young diagram $\lambda$.
Here, the ratio in the coefficient can be derived from the following identity:
\[
\frac{((n-1) \uparrow \lambda)}{(n \uparrow \lambda)} =
\prod_{i=1}^{d(\lambda)} \frac{n-1-(\lambda_i'-i)}{ n+(\lambda_i-i)},
\]
which is evident from the following observation: when decomposing the product over a Young diagram into hooks, the factor corresponding to the box at the southwestern corner of each hook remains in the numerator, while the factor corresponding to the box at the northeastern corner remains in the denominator.
For example, the factors in the gray box of the Young diagram below remain, and the box at position $(3,4)$ corresponds to the factor $n+(4-3)$ in the denominator.
\[
\begin{ytableau}
 {} & {} & {} &{} & *(gray!30) {} \\
 {} & {} & {} &{} & *(gray!30) {} \\
 {} & *(gray!90){} & *(gray!90){} & *(gray!30) {} \\
 *(gray!90) {}
\end{ytableau}
\]
Here $\lambda=(5,5,4,1)$, and $d(\lambda)=3$. The dark gray boxes are the southwestern corners $(\lambda_i',i)$ of the hooks, $1 \le i \le d(\lambda)$; they contribute the factors $n-1-(\lambda_i'-i)$ in the numerator. The light gray boxes are the northeastern corners $(i,\lambda_i)$; they contribute the factors $n+(\lambda_i-i)$ in the denominator.

Finally, we discuss the inverse of the ascension function.

\begin{Definition} \label{def:function_d}
Let $k$ and $n$ be positive integers such that $k \le n$.
We define the \emph{descension function} $\Lower_{k,n}$ in $\mathbb{C}[S_k]$ by
 \[
\Lower_{k,n}(\sigma)= (\sgn \sigma) \sum_{t=0}^{\fix{\sigma}} \binom{\fix{\sigma}}{t} \frac{1}{n^{\downarrow (k-t)}},
\qquad \sigma \in S_k,
 \]
 where $n^{\downarrow t}$ is the falling factorial defined by
 \[
n^{\downarrow t} = \begin{cases}
 n(n-1) \cdots (n-t+1) & \text{for $t=1,2,\dots$}, \\
1 & \text{for $t=0$}.
\end{cases}
\]
\end{Definition}

Note that the definition of the falling factorial requires the assumption $k \le n$.

\begin{Proposition}
Let $k$ and $n$ be positive integers such that $k \le n$.
Then the following holds:
\begin{equation} \label{eq:Lower_action}
 \Raise_{k,n+1}* \Lower_{k,n} =\delta_{e_k}.
\end{equation}
Moreover, we have
\[
\Wg_{k,n} = \Lower_{k,n} * \Wg_{k,n+1}.
\]
\end{Proposition}

\begin{proof}
The second formula follows from Theorem~\ref{thm:recursive} and \eqref{eq:Lower_action}.
To show the first formula~\eqref{eq:Lower_action}, we extend the parameter range of the functions $\G_{k,n}$, $\Wg_{k,n}$, and~$\Raise_{k,n}$
from a positive integer~$n$ (with $k \le n$) to a complex number $z$, following~\cite{CollinsSniady2006}.
We define $\G_{k,z}$ by
\[
\G_{k,z} (\sigma) := z^{\kappa(\sigma)}, \qquad \sigma \in S_k.
\]
This function satisfies the following symmetry:
\[
\G_{k,-z}(\sigma)= (-1)^{\kappa(\sigma)} \G_{k,z}(\sigma)
=(-1)^k \sgn (\sigma) \G_{k,z}(\sigma),
\]
since $\sgn \sigma=(-1)^{k-\kappa(\sigma)}$.
It is known that $\G_{k,z}$ is invertible in the group algebra $\mathbb{C}[S_k]$ if and only if $z \not\in \{0, \pm 1, \dots, \pm(k-1)\}$.
This fact can be confirmed by \eqref{eq:Wg_char}: $\G_{k,z}$ is invertible if and only if $(z \uparrow \lambda) \neq 0$ for every partition $\lambda$ of $k$, and the zeros of $z \mapsto (z \uparrow \lambda)$ are the negatives of the contents $j-i$ of the cells $(i,j)$ of $\lambda$, which range over $\{-(k-1),\dots,k-1\}$ as $\lambda$ varies over all partitions of $k$. For instance, $\G_{2,-1}=\left(\begin{smallmatrix} 1 & -1 \\ -1 & 1\end{smallmatrix}\right)$ is singular.
In this case, we denote by $\Wg_{k,z}$ the inverse of $\G_{k,z}$.
It also satisfies the symmetry relation
$\Wg_{k,-z}(\sigma)=(-1)^k \sgn (\sigma) \Wg_{k,z}(\sigma)$.

Similarly, we define
\[
\Raise_{k,z}(\sigma) := (\G_{k,z-1} * \Wg_{k,z})(\sigma) =
\sum_{t=0}^{\fix{\sigma}} (-1)^{k-t} \binom{\fix{\sigma}}{t} \frac{1}{z^{\uparrow (k-t)}},
\]
based on \eqref{eq:as_g_Wg}.
This function is invertible if and only if $z \not\in \{1, \dots, k\}$.
In particular, if $k \le n$, then $\Raise_{k,n+1}$ is invertible.
We denote the inverse of $\Raise_{k,z+1}$ by $\Raise^{-1}_{k,z+1}$.
Since $\G_{k,z}$ and $\Wg_{k,z}$ are inverses of each other for all $z \not\in \{0, \pm 1, \dots, \pm(k-1)\}$ (note that this set is symmetric under $z \mapsto -z$, so that $\Wg_{k,-z}$ is defined whenever $\Wg_{k,z}$ is),
by using symmetries, we can transform it as follows:
\begin{align*}
\Raise^{-1}_{k,z+1} (\sigma)&= (\Wg_{k,z} * \G_{k,z+1}) (\sigma)
 = (\sgn \sigma) (\Wg_{k,-z} * \G_{k,-z-1})(\sigma) = (\sgn \sigma) \Raise_{k,-z}(\sigma).
\end{align*}
Therefore, the inverse of $\Raise_{k,n+1}$ is $\sgn * \Raise_{k,-n}$,
which equals $\Lower_{k,n}$. Indeed, we have
\begin{align*}
\sgn(\sigma) \Raise_{k,-n}(\sigma)
&=\sgn (\sigma)
\sum_{t=0}^{\fix{\sigma}} (-1)^{k-t} \binom{\fix{\sigma}}{t} \frac{1}{(-n)^{\uparrow (k-t)}} \\
&= \sgn (\sigma)
\sum_{t=0}^{\fix{\sigma}} \binom{\fix{\sigma}}{t} \frac{1}{n^{\downarrow (k-t)}} = \Lower_{k,n}(\sigma).\tag*{\qed}
\end{align*}\renewcommand{\qed}{}
\end{proof}

\begin{Remark}
 We have proved Theorem~\ref{thm:recursive} using random matrices; however, since the theorem is an algebraic statement, it would be interesting to look for an algebraic proof.
\end{Remark}

\subsection{Weingarten functions in the case of low dimensions} \label{subsec:low_dim}

In the previous discussions, we often assumed that $k \le n$.
This is because the function $\G_{k,n}(\sigma)=n^{\kappa(\sigma)}$ on $S_k$ is invertible only when $k \le n$, and the inverse element is $\Wg_{k,n}$.
The Weingarten calculus itself \eqref{eq:Wg_F} holds even when $k > n$; however, in that case, it is necessary to replace $\Wg_{k,n}$ with a slightly different function $w_{k,n}$.
The function $w_{k,n}$ on $S_k$ is a class function satisfying the following equation:
\begin{equation} \label{eq:pseudo}
\G_{k,n} * w_{k,n} * \G_{k,n} = \G_{k,n}.
\end{equation}
Such functions $w_{k,n}$ are not uniquely defined if $k >n$, but they are also called Weingarten functions; see, e.g., \cite{CollinsMatsumoto2009}.
When $k \le n$, $w_{k,n}$ is uniquely determined and coincides with $\Wg_{k,n}$.

\begin{Remark}
Let us assume that $k>n$.
In addition to \eqref{eq:pseudo},
if we have
$w_{k,n} *\G_{k,n} *w_{k,n} = w_{k,n}$,
then $w_{k,n}$ is uniquely determined.
(In \cite{MatsumotoCSS}, it is claimed that $w_{k,n}$ is uniquely determined by condition \eqref{eq:pseudo} alone; however, uniqueness holds only when the additional condition
$w_{k,n} * \G_{k,n} * w_{k,n} = w_{k,n}$
is also satisfied.)
If we denote such a $w_{k,n}$ specifically as $W_{k,n}$, then $W_{k,n}$ is concretely given by
\[
W_{k,n}=\frac{1}{k!} \sum_{\substack{\lambda \vdash k \\ \ell(\lambda) \le n}} \frac{f^\lambda}{(n \uparrow \lambda)} \chi^\lambda,
\]
which imposes a slight restriction $\ell(\lambda) \le n$ on \eqref{eq:Wg_char};
 see \cite[Proposition 2.3]{CollinsSniady2006}.
Here, $\ell(\lambda)=|\{i \ge 1 \mid \lambda_i>0\}|$ is the length of a partition $\lambda=(\lambda_1,\lambda_2,\dots)$.
As already mentioned, the Weingarten formula \eqref{eq:Wg_F} holds for any $w_{k,n}$ satisfying \eqref{eq:pseudo}, not only for this particular $W_{k,n}$.
For more details, see \cite{CollinsFukudaMatsumoto}.
\end{Remark}

Unfortunately, Theorem~\ref{thm:recursive} does not hold for functions $w_{k,n}$ when $k \ge n$.
Consider the case $k=n=2$ and $\sigma =e_2 \in S_2$.
Formally applying Theorem~\ref{thm:recursive} in this case suggests the following:
\begin{equation} \label{eq:wrong_formula}
w_{2,2}(e_2) \stackrel{?}{=} (\Raise_{2,2} * w_{2,1})(e_2) = \sum_{\pi \in S_2} \Raise_{2,2}(\pi) w_{2,1}(\pi).
\end{equation}
It follows that $w_{2,2}(e_2)=\Wg_{2,2}(e_2)= \frac{1}{3}$,
whereas
Definition~\ref{def:function_a} and \eqref{eq:pseudo} imply
\[
\sum_{\pi \in S_2} \Raise_{2,2}(\pi) w_{2,1}(\pi)= \frac{1}{6}
\sum_{\pi \in S_2} w_{2,1}(\pi)= \frac{1}{12}.
\]
Therefore, in the case $k \ge n$, a different equation analogous to Theorem~\ref{thm:recursive} must be found.
We do not know whether such an equation exists in general; however, when $k=n$,
a similar expression to Theorem~\ref{thm:recursive} can be obtained as follows.

\begin{Theorem} \label{thm:dim_equal_degree}
For $\pi \in S_n$, we have
\[
\Wg_{n,n}(\pi)= \sum_{\tau \in S_{n-1}}
a_n(\pi, \tau) \Wg_{n-1,n-1}(\tau),
\]
where the function $a_n\colon S_n \times S_{n-1} \to \mathbb{C}$ is defined by
\begin{equation} \label{eq:def_new_a}
a_n(\pi, \tau)= \sum_{t=0}^{\fix{\pi, \tau}} (-1)^{n-t+1}
\binom{\fix{\pi, \tau}}{t} \frac{1}{n^{\uparrow (n-t)}},
\end{equation}
and $\fix{\pi,\tau}= |\{i \in \{1,2,\dots, n-1\} \mid
\pi (\tau(i))=i \}|$.
\end{Theorem}

Note that $a_n(\pi, \tau)$ differs slightly from $\Raise_{n,n}(\pi \tau)$.

According to this theorem, the incorrect formula \eqref{eq:wrong_formula} can be corrected as follows:
\[
\Wg_{2,2}(e_2)=a_2(e_2,e_1) \Wg_{1,1}(e_1),
\]
and both values are equal to $\frac{1}{3}$.
As another example, the following equation holds:
\[
\Wg_{3,3}(e_3)= a_3(e_3, e_2) \Wg_{2,2}(e_2)+
a_3(e_3, (1 \, 2)) \Wg_{2,2}
((1 \, 2)).
\]
By substituting specific values, this equation becomes:
\[
\frac{7}{120}= \frac{11}{60} \cdot \frac{1}{3} + \frac{1}{60} \cdot \left(- \frac{1}{6}\right).
\]

\begin{proof}[Proof of Theorem~\ref{thm:dim_equal_degree}]
The proof is almost the same as that of Theorem~\ref{thm:recursive}.
We will focus on explaining the different parts.
Let the matrices $U=(u_{ij})$, $R=(r_{ij})$, and $V=(v_{ij})$ be as stated in Lemma~\ref{lem:BNN}.
Let $\pi \in S_n$, and consider
 \[
 \Wg_{n,n}(\pi) = \mathbb{E} \bigl[ u_{1,1} \cdots u_{n-1,n-1} u_{n,n}
 \overline{u_{\pi(1),1} \cdots u_{\pi(n-1),n-1} u_{\pi(n),n}} \bigr].
 \]
 We substitute
 \[
 u_{ij} =
 \begin{cases}
 \displaystyle \sum_{p=1}^{n-1} r_{i,p} v_{p,j} & \text{if $1 \le i \le n$ and $1 \le j \le n-1$}, \\
 r_{i,n} & \text{if $1 \le i \le n$ and $j=n$}.
 \end{cases}
 \]
 After some direct calculations, we have
 \[\Wg_{n,n}(\pi)= \sum_{\tau \in S_{n-1}} \tilde{a}(\pi,\tau) \Wg_{n-1,n-1}(\tau),
 \]
 where
 \[
 \tilde{a}(\pi,\tau)= \sum_{p_1,\dots,p_{n-1}=1}^{n-1} \mathbb{E}\bigl[ r_{1, p_1} \cdots r_{n-1, p_{n-1}} r_{n,n}
 \overline{r_{\pi (\tau(1)), p_1} \cdots r_{\pi (\tau(n-1)), p_{n-1}} r_{\pi(n),n}}
 \bigr].
 \]
 As with the proof of Theorem~\ref{thm:recursive}, we obtain the following expression:
 \[
 \tilde{a}(\pi,\tau)= \mathbb{E} \Biggl[\prod_{i=1}^{n-1}\bigl(\delta (i,\pi(\tau(i)))-x_{i} \overline{x_{\pi(\tau(i))}}\bigr) \times x_n \overline{x_{\pi(n)}} \Biggr].
 \]
 From here, we mimic the proof of Lemma~\ref{lem:expression_A}.
 Define
 \[
 F=\{i \in \{1,2,\dots,n-1\} \mid \pi(\tau(i))=i\},
 \]
 and put $F'=\{1,2,\dots,n\} \setminus F$.
 Note that $F'$ is not the complement in $\{1,2,\dots,n-1\}$, but rather in $\{1,2,\dots,n\}$.
 Then, we have:
 \begin{align*}
 \tilde{a}(\pi,\tau)
 &= (-1)^{n-1} \mathbb{E} \biggl[ \prod_{i \in F} (|x_i|^2-1) \cdot \prod_{j \in F'} |x_j|^2 \biggr] \\
 &= \sum_{s=0}^{|F|} (-1)^{n-1+|F|-s} \sum_{\{i_1< \cdots<i_s\} \subset F}
 \mathbb{E} \biggl[|x_{i_1}|^2 \cdots |x_{i_s}|^2 \prod_{j \in F'} |x_j|^2 \biggr] \\
 &= \sum_{s=0}^{|F|} (-1)^{n-1+|F|-s} \binom{|F|}{s} \frac{1}{n^{\uparrow (s+n-|F|)}} = \sum_{t=0}^{|F|} (-1)^{n-1+t} \binom{|F|}{t} \frac{1}{n^{\uparrow (n-t)}}
 =a_n(\pi, \tau),
 \end{align*}
 as desired. We used Lemma~\ref{lem:Wg_for_R} in the third equality.
\end{proof}

Finally, let us assume $k<n$ and let $\pi \in S_k$.
By repeatedly applying Theorem~\ref{thm:recursive}, we have
\begin{align*}
\Wg_{k,n}(\pi)={}& \sum_{\sigma_{n-1} \in S_k} \sum_{\sigma_{n-2} \in S_k} \cdots \sum_{\sigma_{k} \in S_k}
\Raise_{k,n} \bigl(\pi\sigma_{n-1}^{-1}\bigr) \Raise_{k,n-1} \bigl(\sigma_{n-1}\sigma_{n-2}^{-1}\bigr) \cdots \\
& \hphantom{\sum_{\sigma_{n-1} \in S_k} \sum_{\sigma_{n-2} \in S_k} \cdots \sum_{\sigma_{k} \in S_k}}{}
 \times
\Raise_{k,k+1} \bigl(\sigma_{k+1}\sigma_{k}^{-1}\bigr) \Wg_{k,k} (\sigma_{k}).
\end{align*}
Furthermore, Theorem~\ref{thm:dim_equal_degree} implies that, for each $\sigma \in S_{k}$
\[
\Wg_{k,k}(\sigma) = \sum_{\tau_{k-1} \in S_{k-1}}
\sum_{\tau_{k-2} \in S_{k-2}} \cdots \sum_{\tau_{2} \in S_2}
a_k(\sigma, \tau_{k-1}) a_{k-1}(\tau_{k-1}, \tau_{k-2}) \cdots a_2(\tau_{2}, e_1).
\]
By combining these results, the Weingarten function $\Wg_{k,n} \in \mathbb{C}[S_k]$ for $k \le n$ can be expressed in terms of two types of functions: $\Raise_{k,*}(\sigma)$ and $a_*(\sigma,\tau)$.

\begin{Proposition} \label{prop:Wg_asasas}
Let $k$ and $n$ be positive integers, and suppose $k \le n$.
Then, for any $\pi \in S_k$, we have
\begin{align*}
\Wg_{k,n}(\pi)={}& \sum \Raise_{k,n} \bigl(\pi\sigma_{n-1}^{-1}\bigr) \Raise_{k,n-1} \bigl(\sigma_{n-1} \sigma_{n-2}^{-1}\bigr) \cdots \Raise_{k,k+1} \bigl(\sigma_{k+1}\sigma_{k}^{-1}\bigr) \\
& \hphantom{\sum}{} \times a_k(\sigma_k, \tau_{k-1}) a_{k-1}(\tau_{k-1}, \tau_{k-2}) \cdots a_2(\tau_{2}, e_1)
\end{align*}
summed over all sequences of permutations
\[
(\underbrace{\sigma_{n-1}, \sigma_{n-2}, \dots, \sigma_{k}}_{n-k}, \tau_{k-1},\tau_{k-2},\dots, \tau_2) \in \underbrace{S_k \times \cdots \times S_k}_{n-k} \times S_{k-1} \times S_{k-2} \times \cdots \times S_2.
\]
Here, $\Raise_{k,n}(\sigma)$ is defined in Definition~{\rm \ref{def:function_a}}, and $a_n
(\sigma,\tau)$ is defined in \eqref{eq:def_new_a}.
\end{Proposition}

This proposition allows the Weingarten function to be calculated sequentially starting with the base case $n=1$.

\begin{Example}
Let $k=3$ and $n=4$. Proposition~\ref{prop:Wg_asasas} reads: for $\pi \in S_3$,
\[
\Wg_{3,4}(\pi) = \sum_{\sigma_3 \in S_3} \sum_{\tau_2 \in S_2}
\Raise_{3,4}\bigl(\pi \sigma_3^{-1}\bigr) a_3(\sigma_3,\tau_2) a_2(\tau_2, e_1),
\]
a sum of $12$ terms. Carrying out the computation, one finds
\[
\Wg_{3,4}(e_3)=\frac{7}{360}, \qquad
\Wg_{3,4}((1\,2))=-\frac{1}{180}, \qquad
\Wg_{3,4}((1\,2\,3))=\frac{1}{360},
\]
in agreement with the values obtained from \eqref{eq:Wg_char}.
\end{Example}

\subsection*{Acknowledgements}

B.C.\ is supported by JSPS Grant-in-Aid for Scientific Research (A) no. 25H00593, and Challenging Research (Exploratory) nos. 20K20882 and 23K17299.
S.M.\ is supported by JSPS Grant-in-Aid for Scientific Research (C) no. 22K03233.
The idea to work on this project arose from discussions between B.C.\ and Sasha Bufetov, and from a talk by Ashkan Nikeghbali at the June 2023 RIMS conference on ``Random Matrices and Applications''.
We are grateful for these stimulating conversations and the support provided by RIMS.
We are grateful to the anonymous reviewers for their constructive feedback, which significantly improved this manuscript.

\pdfbookmark[1]{References}{ref}
\LastPageEnding

\end{document}